 \documentclass[11pt]{lms}

\usepackage{amsrefs}

\usepackage{epsfig}  		
\usepackage{epic,eepic}       
\usepackage{graphicx}

\usepackage{enumerate}
\usepackage{mathbbol}
\usepackage{amssymb}
\usepackage{braket}
\usepackage{longtable}
\usepackage[colorlinks]{hyperref}
\usepackage{caption}
\usepackage{mathtools}
\usepackage{microtype}
\usepackage{lineno}

\DeclareSymbolFontAlphabet{\mathbb}{AMSb}
\DeclareSymbolFontAlphabet{\mathbbl}{bbold}

\newtheorem{lemma}{Lemma}[section]
\newtheorem{theorem}[lemma]{Theorem}
\newtheorem{corollary}[lemma]{Corollary}
\newtheorem{proposition}[lemma]{Proposition}

\newtheorem{fact}[lemma]{Fact}

\newtheorem{remark}[lemma]{Remark}

\newtheorem{question}{Question}

\newtheorem{example}{Example}
\newtheorem{definition}[lemma]{Definition}

\numberwithin{equation}{section}

\newcommand{\Z}{\mathbb{Z}}
\newcommand{\R}{\mathbb{R}}  
\newcommand{\C}{\mathbb{C}}
\newcommand{\Q}{\mathbb{Q}}

\newcommand{\bs}{\backslash}

\newcommand\CC{{\mathcal C}}

\newcommand\FF{{\mathcal F}}

\newcommand\II{{\mathcal I}}
\newcommand\JJ{{\mathcal J}}

\newcommand\LL{{\mathcal L}}

\newcommand\<{\langle}
\renewcommand\>{\rangle}

\newcommand{\tp}{\mathrm{tp}}

\newcommand{\Ra}{\Rightarrow}

\newcommand{\La}{\Leftarrow}

\def\Ind#1#2{#1\setbox0=\hbox{$#1x$}\kern\wd0\hbox to 0pt{\hss$#1\mid$\hss}
\lower.9\ht0\hbox to 0pt{\hss$#1\smile$\hss}\kern\wd0}

\def\notind#1#2{#1\setbox0=\hbox{$#1x$}\kern\wd0
\hbox to 0pt{\mathchardef\nn=12854\hss$#1\nn$\kern1.4\wd0\hss}
\hbox to 0pt{\hss$#1\mid$\hss}\lower.9\ht0 \hbox to 0pt{\hss$#1\smile$\hss}\kern\wd0}

\def\includeE#1{{\lhook\kern-3.5pt\joinrel\smash{
    \mathop{\longrightarrow}\limits^{#1}}}}

\def\efor/{Example~\ref{E4}}

\def\BL/{Baldwin--Lachlan}
\def\Bu/{Buechler}
\def\Hr/{Hrushovski}
\def\lm/{locally modular}
\def\wm/{weakly minimal}
\def\nm/{non--modular}
\def\ss/{superstable}
\def\ud/{unidimensional}
\def\sm/{strongly minimal}

\def\abar{\bar{a}}

\def\bbar{\bar{b}}
\def\cbar{\bar{c}}
\def\dbar{\bar{d}}

\def\hbar{\bar{h}}

\def\mbar{\bar{m}}
\def\nbar{\bar{n}}

\def\pbar{\bar{p}}

\def\xbar{\bar{x}}
\def\ybar{\bar{y}}
\def\zbar{\bar{z}}

\def\tr/{trivial}
\def\nt/{non--trivial}
\def\st/{strong type}

\def\abar{\bar{a}}
\def\bbar{\bar{b}}
\def\cbar{\bar{c}}
\def\dbar{\bar{d}}

\def\C{{\mathfrak  C}}

\def\Q{{\mathbb Q}}

\def\Z{{\mathbb Z}}

\def\Fa0{{\FF^a_{\aleph_0}}}

\def\<{\langle}
\def\>{\rangle}

\usepackage{lineno}
\newcommand\myrestriction{\mathord\restriction}
\def\mr#1{\myrestriction_{#1}}
\renewcommand{\C}{\mathfrak{C}}

\title{Indiscernibles in monadically NIP theories}
 \author{Samuel Braunfeld and Michael C. Laskowski}
 \classno{03C45}
 \extraline{This paper is part of a project that has received funding from the 
	European Research Council (ERC) under the European Union's Horizon 2020 
	research and innovation programme (grant agreement No 810115 - Dynasnet). The first author is further supported by Project 24-12591M of the Czech Science Foundation (GA\v{C}R), and supported partly by the long-term strategic development financing of the Institute of Computer Science (RVO: 67985807). The second author is partially supported by NSF grant DMS-2154101.}

 \usepackage{mathrsfs}

\def\J{{\mathcal J}}

\begin{document}
\maketitle
	\begin{abstract}
		We prove various results around indiscernibles in monadically NIP theories. First, we provide several characterizations of monadic NIP in terms of indiscernibles, mirroring previous characterizations in terms of the behavior of finite satisfiability. Second, we study (monadic) distality in hereditary classes and complete theories. Here, via finite combinatorics, we prove a result implying that every planar graph admits a distal expansion. Finally, we prove a result implying that no monadically NIP theory interprets an infinite group, and note an example of a (monadically) stable theory with no distal expansion that does not interpret an infinite group.	
	\end{abstract}

\section{Introduction}
In the setting of hereditary classes of relational structures, the standard model-theoretic condition of an NIP theory collapses to the much stronger \emph{monadically NIP} theory, i.e. a theory that does not have the independence property under arbitrary expansions by unary predicates. Thus monadic NIP is conjectured to be an important dividing line for various combinatorial problems, such as the algorithmic tractability of first-order model checking \cite{dreir2024flip}. The manipulation of indiscernible sequences is a prominent thread in results involving monadic NIP, and this paper consists of three independent sections concerning this. The first section shows that we may provide various characterizations of monadic NIP in terms of indiscernible sequences. The second section describes  how monadic NIP interacts with distality, another model-theoretic condition characterized by the behavior of indiscernibles. One motivation for considering distality is that it also has combinatorial implications, such as the strong Erd\H os-Hajnal property (an improved form of Ramsey's theorem). Another motivation comes from the fact that in monadically NIP theories, every global invariant 1-type is either generically stable or distal, so one might hope to understand monadic NIP by understanding separately how it interacts with stability and with distality. Its interaction with stability leads to the relatively well-understood property of monadic stability, so we begin the study of its interaction with distality here. Last, a recurring theme in model theory is understanding theories by the algebraic objects, particularly groups, that can appear in them, and in the third section we show that monadically NIP theories can encode almost no algebraic structure.

In the first section, we provide various characterizations of monadic NIP in terms of indiscernible sequences, paralleling the characterizations provided in \cite{MonNIP} in terms of finite satisfiability. 

\begin{theorem} [(Theorem \ref{thm:mNIPchar})] \label{thm:mNIPcharintro}
	Let $T$ be a complete theory. Then the following are equivalent.
	\begin{enumerate}[(1)]
		\item $T$ is monadically NIP.
		\item $T$ admits widening of indiscernibles.
		\item $T$ is dp$^+$-minimal.
		\item $T$ is dp-minimal and has endless indiscernible triviality.
	\end{enumerate}
\end{theorem}

There are seemingly good reasons to desire the characterizations in terms of indiscernibility: finite satisfiability requires working over an infinite model and so seems less suited to working with classes of finite structures, where monadic stability and monadic NIP have been intensively studied. Also, one of the issues encountered in \cite{EMonNIP} studying monadic NIP in the setting of existentially closed models is that finite satisfiability is not very well-behaved there, since a finitely satisfiable type cannot necessarily be extended to one over a larger base. However, the Erd\H os-Rado theorem still allows us to obtain and manipulate indiscernible sequences there.

In the second part of this section, we show that dp$^+$-minimality, as opposed to dp-minimality, ensures that a singleton can cut an indiscernible sequence in at most one place.   This clarifies how models of a monadically NIP theory can be decomposed over an indiscernible sequence, which leads to a notion of independence paralleling finite satisfiability.

In the next section, we turn to distality in monadically NIP theories. We begin by characterizing monadic distality in terms of avoiding totally indiscernible sets of singletons (Proposition \ref{prop:distobst}) and showing that monadic distality is equivalent to monadic NIP and distality (Corollary \ref{cor:distcoll1}), mirroring a characterization of monadic stability. One of the main points of \cite{EMonNIP} is that universal theories are the right setting for many monadic model-theoretic properties, so we then study (monadic) distality in universal theories and hereditary classes of finite structures. One of the main results here is that every graph class with bounded twin-width admits a distal expansion by a linear order.

\begin{theorem} [(Theorem \ref{thm:distexp})] \label{thm:distexpintro}
	Let $\CC$ be a hereditary class of binary relational structures with bounded twin-width (for example, a planar graph class). Then $\CC$ admits a distal expansion by a linear order, and thus so does every model of $Th(\CC)_\forall$.
\end{theorem}

Theorem \ref{thm:distexpintro} produces distal expansions of infinite structures without having to analyze definability or indiscernible sequences in those structures, instead relying purely on finite combinatorics. As an example corollary, we obtain the following result about distal expansions of infinite graphs, which on its face is not connected with monadic NIP or with hereditary classes of finite structures.

\begin{corollary}[(Corollary \ref{cor:planar})]
	Let $G$ be an infinite planar graph. Then $G$ admits a distal expansion by a linear order.
\end{corollary}

In the final section, we show that monadically NIP theories cannot interpret an infinite group, although we prove a more general result.  

\begin{proposition}  [(Proposition \ref{prop:nointerp})] \label{prop:intronointerp}
	Let $T$ be a theory with finite dp-rank and endless indiscernible triviality. Then $T$ does not trace-define an infinite cancellative magma.
\end{proposition}

This is part of a project set out in \cite{trace} to identify some notion of triviality and a corresponding Zilber dichotomy for NIP theories based on trace defining an infinite group. Endless indiscernible triviality may be an appropriate such notion in a restricted setting.

Returning to distality, the most commonly known reason for an NIP theory to not admit a distal expansion is if it interprets an infinite field of finite characteristic (e.g., see \cite{aschenbrenner2022distality}). As a corollary of Proposition \ref{prop:intronointerp}, we record the fact that there is a (monadically) stable theory with no distal expansion that does not interpret an infinite group (Corollary \ref{cor:nodist}).

\section{Indiscernibles and naming parameters}
\subsection{Characterizations of monadic NIP via indiscernibles}

Before stating the relevant definitions, we give some context for this subsection, where the goal is to prove Theorem \ref{thm:mNIPcharintro}. The point of view resulting from this theorem is that we may decompose models of a monadically NIP theory into an indiscernible sequence, rather than into the $M$-f.s. (short for \emph{finitely satisfiable}) sequences used in \cite{MonNIP}. In particular,  by iterating the fact that $T$ admits widening of indiscernibles, it will follow that given any $M \prec \C$ (where $\C$ is the monster model) and Dedekind complete indiscernible sequence $\II = (\abar_i : i \in I)$ in $\C$, we may ``widen'' $\II$ to an indiscernible sequence $\II^+ = (\abar_i\bbar_i : i \in I)$ such that $M \subset \bigcup \II^+$. This is essentially the point of view given on monadically NIP theories in \cite{blumensath2011simple}, which also proves that monadically NIP theories admit widening of indiscernibles. Then dp$^+$-minimality is somewhat analogous to the f.s. dichotomy of \cite{MonNIP}, describing how a single additional point interacts with a partial decomposition. Widening of indiscernibles corresponds to \cite[Lemma 3.2]{MonNIP} that any partial decomposition can be extended to contain any additional point, and similarly iterating gives the statement that a partial decomposition can always be extended to contain any model. The last characterization is analogous to \cite[Lemma 3.3]{MonNIP} that dependence is trivial and transitive. Dp-minimality may be more transparently viewed as a transitivity statement via the following characterization: if there are indiscernible sequences $\II, \JJ$ and a singleton $b$ such that neither $\II$ nor $\JJ$ are indiscernible over $b$, then $\II$ and $\JJ$ are not mutually indiscernible.

We now define the notions appearing in Theorem \ref{thm:mNIPcharintro}. All indiscernible sequences will be assumed infinite. One can view dp-minimality as describing how  an indiscernible sequence $\II=(\abar_i:i\in I)$ can be `injured' by the addition of parameters.  

\begin{definition} \label{def:dpmin}
	A theory $T$ is \emph{dp-minimal} if for every indiscernible sequence $\II = (\abar_i : i \in \Q)$ and every singleton $b$, there is $i^* \in \R \cup \set{\pm \infty}$ such that $\II_{<i^*}$ and $\II_{>i^*}$ are mutually indiscernible over $b$.
	
	A theory $T$ is \emph{dp$^+$-minimal} if we may instead choose $i^* \in \R \cup \set{\pm \infty}$ so that  $\II_{<i^*}$ and $\II_{>i^*}$ are mutually indiscernible over $\abar_{i^*}b$ if $i^* \in \Q$, and over $b$ otherwise.
\end{definition}

The definition of dp-minimality above is equivalent to the characterization given in \cite[Theorem 4.18 $(iii)_2$]{Guide} (with $\kappa =2$ and $p$ the partial type $\set{x=x}$). The characterization works with an arbitrary infinite index set and allows one of $\II_{<i^*}$ and $\II_{>i^*}$ to be finite, in which case we only demand that the other is indiscernible over $b$. These are clearly equivalent on $\Q$-indexed sequences, so the characterization implies our definition. It suffices to show our definition implies the characterization on countable sequences, which can always be extended to $\Q$-indexed sequences, where our definition applies and produces a cut $i^*$ that also works for the characterization.

The name dp$^+$-minimality is based on the notion of strongly$^+$-dependent theories introduced in \cite{strongdep} (see Conclusion 2.11 there) and of dp$^+$-rank mentioned in \cite[Remark 4.19]{Guide}.

\begin{definition}
	A theory $T$ \emph{admits widening of indiscernibles} if for every infinite indiscernible sequence $\II=(\abar_i:i\in I)$ and every singleton $b\in\C$,
	there is a one-point extension $(J,\le)\supseteq (I,\le)$, an extension $(\abar_j:j\in J)$ of $\II$, an element $j^*\in J$ and a sequence $(b_j:j\in J)$ such that $b_{j^*}=b$
	and $(\abar_j b_j:j\in J)$ is indiscernible.
\end{definition}

\begin{definition}
    A theory $T$ has \emph{endless indiscernible triviality} if for every indiscernible sequence $\II = (\abar_i : i \in I)$ without endpoints, and for every set $B$, if $\II$ is indiscernible over each $b \in B$ then $\II$ is indiscernible over $B$.

    A theory has \emph{indiscernible triviality} if the statement above holds for arbitrary infinite indiscernible sequences (possibly with endpoints).
\end{definition}

We  recall the f.s. dichotomy and some related notions from \cite{MonNIP}.

\begin{definition} $T$ has the {\em f.s.\ dichotomy} if, for all models $M$, all finite tuples $\abar,\bbar \in \C$, if $\tp(\bbar/M\abar)$ is finitely satisfied in $M$, then for any singleton $c$,
	either $\tp(\bbar/M\abar c)$ or $\tp(\bbar c/M\abar)$ is finitely satisfied in $M$.
\end{definition}

	\begin{definition} \label{AJ} Let $M$ be a model, let $C \supseteq M$, and let  $(I,\le)$ be any linearly ordered index set.
	\begin{itemize}
		\item  Suppose $\<A_i:i\in I\>$ is any sequence of sets, indexed by $(I,\le)$.  For $J\subseteq I$, $A_J$ denotes $\bigcup_{j\in J} A_j$, and for $i^*\in I$,
		$A_{<i^*}$ denotes $\bigcup_{i<i^*} A_i$.  $A_{\le i^*}$ and $A_{>i^*}$ are defined analogously.
		\item  For $C\supseteq M$, an {\em $M$-f.s.\ sequence over $C$},  is a  sequence of sets $\<A_i:i\in I\>$ such that $\tp(A_i/A_{<i}C)$ is finitely satisfied in $M$ for every $i\in I$.
		When $C=M$ we simply say $\<A_i:i\in I\>$ is an $M$-f.s.\ sequence.
	\end{itemize}
\end{definition}

	\begin{definition}
	We call $C\supseteq M$ {\em full} if, for every $n$, every $p\in S_n(M)$ is realized in $C$.
\end{definition}

The relevance of fullness is the following stationarity result.

\begin{lemma}[({\cite[Lemma 2.12]{MonNIP}})]
Suppose $C\supseteq M$ is full and $p\in S(C)$ is finitely satisfied in $M$.  Then for any set $D\supseteq C$,
there is a unique $q\in S(D)$ extending $p$ that remains finitely satisfied over $M$.
\end{lemma}

\begin{definition}
	Fix any full $C \supset M$ and $q(\xbar) \in S(C)$ finitely satisfied in $M$. For any linear order $(I,\le)$ (even finite)
	the {\em $I$-Morley product $q^{(I)}$} is  the unique type in the variables $(\xbar_i:i\in I)$ satisfying
	$q = \tp(\xbar_i/C)$ and $\tp(\xbar_i/C\xbar_{<i})$ is finitely satisfiable in $M$ for every $i\in I$.
\end{definition}

\begin{lemma}  \label{r}  Suppose $C\supseteq M$ is full and $r(\xbar,\ybar)\in S(C)$ is finitely satisfied in $M$, and let $q(\xbar)$ be the restriction of $r(\xbar,\ybar)$ to $\xbar$.
	For any linear order $(I,\le)$ and any $(\abar_i:i\in I)$ realizing $q^{(I)}$, there are $(\bbar_i:i\in I)$ such that $(\abar_i\bbar_i:i\in I)$ realizes $r^{(I)}$.
\end{lemma}

\begin{proof}  Choose any $(\abar_i'\bbar_i':i\in I)$ realizing $r^{(I)}$.  Then $(\abar_i':i\in I)$ also realizes $q^{(I)}$, so there is an automorphism
	$\sigma\in Aut(\C)$ fixing $C$ pointwise with $\sigma(\abar_i')=\abar_i$.  Then $(\abar_i\sigma(\bbar_i'):i\in I)$ also satisfies $r^{(I)}$.
\end{proof}

\begin{lemma} \label{useful} Suppose $M$ is a model and $A,B,C$ are possibly infinite sets.
If $\tp(A/MC)$ is finitely satisfied in $M$ then there is $C^*$ such that 
\begin{enumerate}
\item  $\tp(C^*/MA)=\tp(C/MA)$; and
\item  $\tp(AB/MC^*)$ is finitely satisfied in $M$.
\end{enumerate}
\end{lemma}

\begin{proof}  Fix enumerations $\abar$ of $A$ and $\bbar$ of $B$, and choose (disjoint) tuples of variable $\xbar^*,\ybar^*$ with
$\lg(\xbar^*)=\lg(\abar)$ and $\lg(\ybar^*)=\lg(\bbar)$.  Put
$$\Gamma(\xbar^*,\ybar^*):=\tp(\abar/MC)\cup\tp(\abar\bbar/M)$$
which is a set of formulas over $MC$.

\medskip
\noindent{\bf Claim.}  $\Gamma(\xbar^*,\ybar^*)$ is finitely satisfied in $M$.

\begin{proof}  We may assume $\Gamma$ is closed under finite conjunctions, so choose any $\phi(\xbar,\mbar,\cbar)\in\tp(A/MC)$ and
$\psi(\xbar,\ybar,\mbar)\in\tp(\abar\bbar/M)$.  We must find $(\mbar'\nbar')\in M^{\lg(\xbar)+\lg(\ybar)}$ realizing $\phi(\xbar,\mbar,\cbar)\wedge\psi(\xbar,\ybar,\mbar)$.
To find these, consider $$\theta(\xbar,\mbar,\cbar):=\phi(\xbar,\mbar,\cbar)\wedge\exists\ybar\psi(\xbar,\ybar,\mbar)$$
Note that $\theta(\xbar,\mbar,\cbar)\in\tp(\abar/MC)$.  Since $\tp(A/MC)$ is finitely satisfied in $M$, there is $\mbar'\in M^{\lg(\xbar)}$ such that
$\theta(\mbar',\mbar,\cbar)$ holds.  In particular, $\exists\ybar\psi(\mbar',\ybar,\mbar)$ holds.  Since $M\preceq\C$, there is $\nbar'\in M^{\lg(\ybar)}$ such that
$\psi(\mbar',\nbar',\mbar)$.  Thus, the tuple $(\mbar',\nbar')$ suffices.
\end{proof}

There is a complete type $q(\xbar^*,\ybar^*)\in S(MC)$ extending $\Gamma(\xbar^*,\ybar^*)$
that is finitely satisfied in $M$ (see, e.g., \cite[Fact 2.3(2)]{MonNIP}).
Choose any realization $(A_1,B_1)$ of $q$.

Since $\tp(A/MC)\subseteq q$, $\tp(A_1/MC)=\tp(A/MC)$, so choose an automorphism $\sigma\in Aut(\C)$ fixing $MC$ pointwise
with $\sigma(A_1)=A$.  Put $B_2:=\sigma(B_1)$.

Note that $AB_2$ also realizes $q$, so since $\tp(AB/M)\subseteq q$, $\tp(AB_2/M)=\tp(AB/M)$.  Hence, we also have $\tp(B_2/AM)=\tp(B/AM)$.
Choose $\delta\in Aut(\C)$ fixing $AM$ pointwise with $\delta(B_2)=B$ and put $C^*:=\delta(C)$.

We claim that this choice of $C^*$ satisfies (i) and (ii) of the Lemma.
For (i), via $\delta$ we have $\tp(C^*/AM)=\tp(C/AM)$.
For (ii),  first note that
$$\tp(ABMC^*)= \tp(AB_2MC)= \tp(A_1B_1MC)$$
with the first equality coming from $\delta$ and the second equality coming from $\sigma$.
However, $A_1B_1$ realizes $q$, which is finitely satisfied over $MC$.
Thus, by the equalities above, $\tp(AB/MC^*)$ is finitely satisfied as well.
\end{proof}

The following lemma is essentially \cite[Theorem 4.14]{blumensath2011simple}, although the f.s. dichotomy allows us to give a simple proof.

\begin{lemma}  \label{1.1}   If $T$ has the f.s. dichotomy then for every infinite indiscernible $(\abar_i:i\in I)$ where $(I, \le)$  is Dedekind complete and for every singleton $b\in \C$,
	there are $(b_i:i\in I)$ and some $i^* \in I$ such that $b_{i^*}=b$ and $(\abar_ib_i:i\in I)$ is indiscernible.
\end{lemma}

\begin{proof}  Choose $M$ and a full $D\supset M$ such that $(\abar_i:i\in I)$ is an $M$-f.s.\ sequence over $D$ and $(\abar_i:i\in I)$
is indiscernible over $D$, by \cite[Lemma 2.20]{MonNIP}.  By Lemma~\ref{useful} and possibly replacing 
replacing $D$ by a conjugate,
we may additionally assume that $\tp(b/D)$ is finitely satisfied in $M$.

	\begin{claim*}
		There is some $i^*\in I$ such that $(\abar_i:i<i^*) ^\smallfrown (\abar_{i^*}b) ^\smallfrown (\abar_i:i>i^*)$ is an $M$-f.s.\ sequence over $D$.
	\end{claim*}

		\emph{Proof of Claim.} This is immediate from the proof of Lemma \cite[Lemma 3.2]{MonNIP}, where the following is shown. Let $I_0$ be the maximal initial segment of $I$ such that $\tp(b/D\abar_{I_0})$ is finitely satisfiable in $M$, and let $I = I_0  ^\smallfrown I_1$. If $I_1$ has a minimal element $i_1^*$, it is shown that we can take $i^* = i_1^*$. Otherwise, it is shown that $(\abar_i:i \in I_0) ^\smallfrown (b) ^\smallfrown (\abar_i:i\in I_1)$ is an $M$-f.s.\ sequence over $D$. In this case, $I_0$ must have a maximal element $i_0^*$ by Dedekind completeness, and $(\abar_i:i \in I_0\bs i_0^*) ^\smallfrown (\abar_{i_0^*}b) ^\smallfrown (\abar_i:i\in I_1)$ is also an $M$-f.s.\ sequence over $D$. \hfill $\lozenge$

	Let $r(\xbar,y):=\tp(\abar_{i^*}b/D)$
	and let $q(\xbar):=\tp(\abar_i/D)$ for some (equivalently, for every) $i\in I$.  Then, by Lemma~\ref{r} there are $\{b_i:i\in I\}$ such that $(\abar_i b_i:i\in I)$ realizes $r^{(I)}$.
	By \cite[Lemma 2.22]{MonNIP}, $(\abar_i b_i:i\in I)$ is indiscernible over $D$.  Finally, choose an automorphism of $\C$ fixing
	$A_ID$ with $\sigma(b_{i^*})=b$.  Then $(\abar_i\sigma(b_i):i\in D)$ is as desired.
\end{proof}

Before the main theorem, we need state the following obstruction to monadic NIP.

\begin{definition} \label{def:coding}
		A {\em tuple-coding configuration} is a formula $\phi(\xbar, \ybar, z)$, a pair of mutually indiscernible sequences $(\abar_i : i \in \Q)$ and $(\bbar_j: j \in \Q)$, and a set of singletons $\set{c_{i,j}| i \in \Q, j \in \Q}$, such that $\models \phi(\abar_i, \bbar_j, c_{k,\ell})$ if and only if $(i,j) = (k, \ell)$.
\end{definition}

\begin{theorem} \label{thm:mNIPchar}
	Let $T$ be a complete theory. Then the following are equivalent.
	\begin{enumerate}[(1)]
		\item $T$ is monadically NIP.
		\item $T$ admits widening of indiscernibles.
		\item $T$ is dp$^+$-minimal.
		\item $T$ is dp-minimal and has endless indiscernible triviality.
	\end{enumerate}
\end{theorem}
\begin{proof}
	$(1) \Rightarrow (2)$ Let $\II = (\abar_i  : i \in I)$ be an infinite indiscernible sequence. Let $(I_2,\le)$ be the Dedekind completion of $(I,\le)$.  By Compactness, choose
	$(\abar_i:i\in I_2)$ extending $(\abar_i:i\in I)$ that is also indiscernible.  By Lemma~\ref{1.1} there is $i^*\in I_2$ and $(b_j:j\in I_2)$ such that $b_{i^*}=b$
	and $(\abar_i b_i:i\in I_2)$ is indiscernible.  Finally, let $(J,\le)\subseteq (I_2,\le)$ be the subordering with universe $I\cup\{i^*\}$.  Then $(\abar_j b_j:j\in J)$ is as desired.
	
	$(2) \Rightarrow (3)$ Suppose $T$  admits widening of indiscernibles,  $\II = (\abar_i : i \in \Q)$ is an infinite indiscernible sequence, and  $b$ is a singleton. Let $\II^+ = (\abar_jb_j : j \in J)$ with $J \subset \R \cup \set{\pm \infty}$ be the sequence produced by widening $\II$, with $b_{j^*} = b$. Then $\II^+_{<j^*}$ and $\II^+_{>j^*}$ are mutually indiscernible over $\abar_{j^*}b$.
	
	$(3) \Rightarrow (1)$ Suppose $T$ is not monadically NIP, so by \cite{MonNIP} $T$ admits a tuple-coding configuration $\II = (\abar_i : i \in \Q)$, $\JJ = (\bbar_j : j \in \Q)$, $\set{c_{i,j} | i,j  \in \Q}$, $\phi(\xbar, \ybar, z)$. Since $\II$ and $\JJ$ are mutually indiscernible, $\II^* = (\abar_i\bbar_i : i \in \Q)$ is also indiscernible. Then any $c_{i,j}$ with $i \neq j$ witnesses the failure of dp$^+$-minimality.
	
	$(1) \Leftrightarrow (4)$ This is contained in the corrigendum to \cite{MonNIP}.
	\end{proof}

 We note one application in a special case.  Call a  theory $T$ {\em binary} if the type of any $n$-tuple of elements from any $M\models T$ is determined by the union of the types of its 2-element subsequences.  Clearly, if $T$ is binary, then it has indiscernible triviality.

 \begin{corollary} \label{cor:binary} If $T$ is binary, 
 then $T$ is monadically NIP if and only if every completion of $T$ is dp-minimal.  In this case, $T$ has indiscernible triviality (as opposed to just endless indiscernible triviality).
 \end{corollary}

\subsection{Uniqueness of the cut point}
In a dp-minimal theory, a singleton can ``exceptionally relate'' to several points of an indiscernible sequence, although removing any one of these points destroys the relation. For example, consider an infinite linearly independent set in a vector space and the sum of some finite subset. However, in a dp$^+$-minimal theory, we will show that singleton can ``exceptionally relate'' to only one cut in an endless indiscernible sequence. This implies if an indiscernible $\II$ is not indiscernible over $b$, then the choice of $i^*$ in Definition \ref{def:dpmin} is essentially unique, and is truly unique if the index sequence is also dense and Dedekind complete.

One outcome of this is that an endless, dense, Dedekind complete indiscernible sequence in a monadically NIP theory gives rise to a well-defined notion of independence (still depending on the sequence). Namely, given a model $M \subset \C$ and such an indiscernible sequence $\II$, then we may ``widen'' $\II$ to $\II^+$ covering $M$. There will be some points that $\II$ remains indiscernible over, which would be inserted into every tuple of $\II^+$, which we call orthogonal points. Then two non-orthogonal points may be considered dependent over $\II$ if they are placed into the same tuple of $\II^+$ and independent otherwise, and orthogonal points are considered independent from everything. From the results of this section, it follows that at every step there is a unique tuple where each non-orthogonal point can be inserted, and furthermore this tuple does not depend on the order in which points are inserted.

The main result concerning indiscernible sequences is Proposition \ref{indiscversion}. The proof of this makes use of the parallel result Proposition \ref{fsversion} concerning $M$-f.s. sequences (Definition \ref{AJ}). We begin with some general lemmas on indiscernibles without any assumption of monadic NIP.

\begin{lemma}  \label{triple}  Suppose $(I,\le)$ is endless and $I=I_0 ^\smallfrown I_1 ^\smallfrown I_2$ is any partition of $I$ into three non-empty convex sets.
	If  both $(\abar_i:i\in I_0)$ and $(\abar_i:i\in I_1\cup I_2)$ are mutually indiscernible,  and $(\abar_i:i\in I_0\cup I_1)$ and $(\abar_i:i\in I_2)$ are mutually indiscernible, then
	$(\abar_i:i\in I)$ is indiscernible.
\end{lemma}
\begin{proof}
    Given an increasing tuple $\abar_{i_1} \dots \abar_{i_n}$, we may shift an individual $\abar_{i_j}$ from $\II_1$ to $\II_0$ while preserving the type by using the second mutual indiscernibility assumption, and similarly may shift an individual $\abar_{i_j}$ from $\II_2$ to $\II_1$ while preserving the type by using the first mutual indiscernibility assumption. Repeating this process, we eventually shift all individual subtuples into $\II_0$, which is assumed to be indiscernible.
\end{proof}

\begin{lemma} \label{cutsingleton}
	Let $\II = (\abar_i : i \in I)$ be a sequence, and suppose there is a cut $I = I_0  ^\smallfrown I_1$ such that $I_0$ and $I_1$ are mutually indiscernible. If there is $i^* \in I$ such that $\II \bs \set{\abar_{i^*}}$ is indiscernible, then $\II$ is indiscernible.
\end{lemma}

\begin{proof}
    Given an increasing tuple $\abar_{i_1} \dots \abar_{i_n}$, if some $\abar_{i_j} = \abar_{i^*}$, we may use the mutual indiscernibility to shift $\abar_{i^*}$ to another tuple on the same side of the cut, possibly also shifting other $\abar_{i_k}$ if this is needed to make space.
\end{proof}

\begin{lemma} \label{bigind}
 Let $\II = (\abar_i : i \in I)$ be an endless indiscernible sequence, and let $b$ be a singleton. If any of the following hold, then  $\II$ is indiscernible over $b$. 
 
	\begin{enumerate}[(1)]
		\item There are two distinct cuts $I = I_0  ^\smallfrown I_1$ and $I = J_0  ^\smallfrown J_1$ such that $(\abar_i : i \in I_0)$ and $(\abar_i : i \in I_1)$ are mutually indiscernible over $b$, as are $(\abar_i : i \in J_0)$ and $(\abar_i : i \in J_1)$.
		\item There is a cut $I = I_0  ^\smallfrown I_1$ such that $(\abar_i : i \in I_0)$ and $(\abar_i : i \in I_1)$ are mutually indiscernible over $b$, and an element $i^* \in I$ that is neither the maximal element of $I_0$ nor the minimal element of $I_1$ with $(\abar_i : i<i^*)$ and $(\abar_i : i>i^*)$ are mutually indiscernible over $b$.
		\item There are two distinct elements $i^*, i' \in I$ with $(\abar_i : i<i^*)$ and $(\abar_i : i>i^*)$ mutually indiscernible over $a_{i^*}b$ and $(\abar_i : i<i')$ and $(\abar_i : i>i')$ mutually indiscernible over $b$, and there is no cut $I = I_0  ^\smallfrown I_1$ such that  $(\abar_i : i \in I_0)$ and $(\abar_i : i \in I_1)$ are mutually indiscernible over $b$ and $i^*, i'$ are the maximal element of $I_0$ and the minimal element of $I_1$.
	\end{enumerate}
\end{lemma}
\begin{proof}
	(1) Suppose $J_0 \subset I_0$. Apply Lemma \ref{triple} to $(J_0  ^\smallfrown I_0) \bs (J_0  ^\smallfrown I_1)$.
	
	(2) Suppose there is such a cut $I = I_0  ^\smallfrown I_1$ and an $i^* \in I$ such that $(\abar_i:i<i^*)$ and $(\abar_i:i>i^*)$ are mutually indiscernible over $b$. Possibly reversing the order, we may suppose $i^* \in I_0$. Let $I\bs \set{i^*} = J_0  ^\smallfrown J_1$ with $J_0 = \set{i \in I | i < i^*}$ and $J_1 = \set{i \in I_0 | i > i^*}$. If $i^*$ is not the maximal element of $I_0$, then this cut is distinct from $(I_0 \bs \set{i^*}) ^\smallfrown I_1$. So by Case (1), $\II \bs \set{\abar_{i^*}}$ is indiscernible over $b$. So by Lemma \ref{cutsingleton}, $\II$ is indiscernible over $b$.
	
	(3) Suppose there are such $i^*$ and $i'$ with $i' < i^*$ (the case $i' > i^*$ is symmetric). Applying Case (2) to $\II \bs \set{\abar_{i^*}}$, we have that $i'$ is the immediate predecessor of $i^*$. Now consider the cut $I = I_0  ^\smallfrown I_1$ with $I_0 = \set{i \in I | i \leq i'}$ and $I_1 =\set{ i \in I | i \geq i^*}$. 
	Routine manipulation of indiscernibles shows $\II\upharpoonright{I_0}$ and $\II\upharpoonright{I_1}$ are mutually indiscernible over $b$.
\end{proof}

\begin{lemma}  \label{f.s.indisc} Suppose $(I,\le)$ is endless, $I_0$ is a proper initial segment, and $I_1$ is a proper final segment (but we are not assuming $I=I_0\cup I_1$).
	If $(\abar_i:i\in I)$ is indiscernible, and $M$ and a full $D \supset M$ are chosen so that $(\abar_i:i\in I_0) ^\smallfrown (Y)  ^\smallfrown (\abar_i:i\in I_1)$ is an $M$-f.s.\ sequence over $D$
	and $(\abar_i:i\in I_0)$, $(\abar_i:i\in I_1)$ are each indiscernible over $D$, then $(\abar_i:i\in I_0)$ and $(\abar_i:i\in I_1)$ are mutually indiscernible over $DY$.
\end{lemma}
\begin{proof}
	This is immediate from \cite[Proposition 2.17]{MonNIP} stating that an $M$-f.s. sequence over a full $D \supset M$ is an order-congruence over $D$ (\cite[Definition 2.16]{MonNIP}).
\end{proof}

We now begin working in monadically NIP theories.

\begin{proposition}[(Cut Lemma, f.s.\ version)]   \label{fsversion}   Let $T$ be monadically NIP. Suppose $(I,\le)$ is endless, $(\abar_i:i\in I)$ is indiscernible, and $b\in \C$ is any singleton for which
	$(\abar_i:i\in I)$ is not indiscernible over $b$.  Let $M\subseteq D$ be such that $M$ is a model, $D/M$ is full, $\tp(b/D)$ is finitely satisfiable in $M$, and
	$(\abar_i:i\in I)$ is an $M$-f.s.\ sequence over $D$.  Let $I_0$ be the maximal initial segment of $I$ such that $\tp(b/A_{I_0}D)$ is finitely satisfied in $M$, and let $I = I_0  ^\smallfrown I_1$. Let $i_0^*$ denote the maximal element of $I_0$ if it exists, and $i_1^*$ denote the minimal element of $I_1$ if it exists.
	\begin{enumerate}[(1)]  
		\item If $i_1^*$ does not exist (and possibly if it does), then 
		$$(\abar_i:i\in I_0)^\smallfrown (b)^\smallfrown (\abar_i:i\in I_1)$$ is an $M$-f.s.\ sequence over $D$
		\item If $i_1^*$ exists, then $$(\abar_i:i<i_1^*)^\smallfrown (\abar_{i_1^*} b)^\smallfrown (\abar_i:i>i_1^*)$$ is an $M$-f.s.\ sequence over $D$.
		\item If $i_0^*$ exists, then 
		$$(\abar_i:i<i_0^*)^\smallfrown (\abar_{i_0^*} b)^\smallfrown (\abar_i:i>i_0^*)$$ is an $M$-f.s.\ sequence over $D$ if and only if 
		$$(\abar_i:i\in I_0)^\smallfrown (b)^\smallfrown (\abar_i:i\in I_1)$$ is an $M$-f.s.\ sequence over $D$.
	\end{enumerate}
	
	Furthermore, these are the only three possibilities for inserting $b$ so that the result is an $M$-f.s. sequence over $D$.
\end{proposition}

\begin{proof}
	The proof of Clauses (1) and (2) is explicitly contained in the proof of \cite[Lemma 3.2]{MonNIP}. The ``if'' direction of Clause (3) is immediate from Condensation (\cite[Definition 2.5]{MonNIP}). 
	
	For the ``only if'' direction of Clause (3),  suppose $b$ can be inserted at $i_0^*$. Then we must show that $(\abar_i:i\in I_0)^\smallfrown (b)^\smallfrown (\abar_i:i\in I_1)$ is an $M$-f.s.\ sequence over $D$. If $i_1^*$ does not exist, this follows from Clause (1). If $i_1^*$ does exist, then Clause (2) gives that $(\abar_i:i<i_1^*)^\smallfrown (\abar_{i_1^*} b)^\smallfrown (\abar_i:i>i_1^*)$ is an $M$-f.s.\ sequence over $D$, and together with our assumption that $(\abar_i:i<i_0^*)^\smallfrown (\abar_{i_0^*} b)^\smallfrown (\abar_i:i>i_0^*)$ is an $M$-f.s.\ sequence over $D$, we are finished.

	We now prove the Furthermore clause. We recall Lemma \ref{f.s.indisc}, which we will use repeatedly without mention. Suppose $b$ can be inserted at some index $i^*$, which possibly represents a new cut. By the definition of $I_0$, we cannot have $i^*$ above any element of $I_1$. If $i^*_1$ exists and $i^* \in I$ with $i^* \neq i^*_1$, then by Lemma \ref{bigind}(3) there is a cut $I = J_0  ^\smallfrown J_1$ such that $i^*$ is the maximal element of $J_0$ and $i^*_1$ is the minimal element of $J_1$, so $I_0 = J_0$.
	
	We cannot have that $i^*_1$ does not exist and $i^* \not\in I$ defines a cut other than $I_0  ^\smallfrown I_1$, since this would contradict Lemma \ref{bigind}(1). The case where $i^*_1$ does not exist and $i^* \in I$ with $i^* \neq i^*_0$ is ruled out by Lemma \ref{bigind}(2), which also rules out case where $i^*_1$ exists and $i^* \not\in I$ defines a cut other than $I_0  ^\smallfrown I_1$.
\end{proof}

The following Corollary further clarifies the situation.

\begin{corollary} \label{special} Let $I,M,D$, $(\abar_i:i\in I)$ and $b\in\C$ be as in  Lemma~\ref{fsversion}.   Suppose  $(\abar_i:i\in I)$ is not indiscernible over $b$ and let $I=I_0^\smallfrown I_1$
	be the cut where $I_0$ is the maximal initial segment of $I$ satisfying $\tp(b/A_{I_0}D)$ is finitely satisfied in $M$.  Then:
	\begin{enumerate}[(1)]
		\item  The only possible 
		$i^*\in I$ such that $(\abar_i:i<i^*)^\smallfrown (\abar_{i^*} b)  ^\smallfrown (\abar_i:i>i^*)$ is an $M$-f.s.\ sequence over $D$ are the maximal element of $I_0$ or the minimal element of $I_1$.
		In particular, there are at most two such $i^*\in I$.
		\item  If $(I,\le)$ is dense, then there is at most one such $i^*\in I$.
		\item  If  $(I,\le)$ is Dedekind complete, then there is at least one such $i^*\in I$.
	\end{enumerate}
\end{corollary}

\begin{proof}  (1) Immediate from the statement of Lemma \ref{fsversion}.

	(2)  If $(I,\le)$ is dense, then it is not possible for $I_0$ to have a maximum and $I_1$ has a minimum, so we are done by (1).
	
	(3)  If $(I,\le)$ is Dedekind complete, then either $i_0^*$ or $i_1^*$ from Lemma \ref{fsversion} must exist. If $i_1^*$ exists, then we are finished by Clause (2) of Lemma \ref{fsversion}. If $i_1^*$ does not exist, then we are finished by Clauses (1) and (3) of Lemma \ref{fsversion}.
\end{proof} 

\begin{remark} {\em   Easy examples involving $T=DLO$ show that the remaining possibilities can indeed happen.
		\begin{itemize}
			\item  If $a_q=q$ for all $q\in\Q$, then for $b=\pi$ and every element of $M$ is larger than every rational,
			then $(a_q:q<\pi)^\smallfrown (\pi)^\smallfrown (a_q:q>\pi)$ is an $M$-f.s.\ sequence, but 
			there is no $q\in \Q$ such that $(a_i:i<q)^\smallfrown (a_qb) ^\smallfrown (a_i:i>q)$ is an $M$-f.s.\ sequence.
			\item  If $a_i=i$ for all $i\in\Z$, $b=\pi$, every element of $M$ is larger than every rational,
			then $(a_i:i<\pi)^\smallfrown (\pi)^\smallfrown (a_i:i>\pi)$ is an $M$-f.s.\ sequence, but also both
			$(a_i:i<3)^\smallfrown (3^\smallfrown b)^\smallfrown (a_i:i\ge 4)$ 
			and 
			$(a_i:i\le 3)^\smallfrown (b^\smallfrown 4)^\smallfrown  (a_i:i> 4)$
			are $M$-f.s.\ sequences.
		\end{itemize}
	}
\end{remark}

\begin{proposition}[(Cut Lemma, indiscernible version)]  \label{indiscversion}  Suppose $T$ is monadically NIP and $(I,\le)$ is endless.  
	Suppose $(\abar_i:i\in I)$ is indiscernible and $b\in \C$ is a singleton such that $(\abar_i : i \in I)$ is not indiscernible over $b$.
	Then one of the following holds.
	\begin{enumerate}[(1)]
		\item  There is a cut $I=I_0^\smallfrown I_1$
		such that $(\abar_i:i\in I_0)$ and $(\abar_i:i\in I_1)$ are mutually indiscernible over $b$.
		
		Furthermore in this case, the cut is unique. And $i^* \in I$ is such that $(\abar_i:i<i^*)$ and $(\abar_i:i>i^*)$ are mutually indiscernible over $b$ if and only if $i^*$ is either the maximal element of $I_0$ or the minimal element of $I_1$.
		\item  There is a unique $i^*\in I$ such that $(\abar_i:i<i^*)$ and $(\abar_i:i>i^*)$ are mutually indiscernible over $b$. 
		
		Furthermore in this case, they are mutually indiscernible over $a_{i^*}b$. 
	\end{enumerate}
\end{proposition}

\begin{proof}   
    	As in the beginning of Lemma \ref{1.1}, choose $M$ and a full $D \supset M$ so that $\tp(b/D)$ is finitely satisfiable in $M$, and $(\abar_i:i\in I)$ is an $M$-f.s.\ sequence over $D$ and is indiscernible over $D$.
	We  apply the f.s.\ Cut Lemma.   If there is a cut $I=I_0^\smallfrown I_1$ such that $(\abar_i:i\in I_0)^\smallfrown(b)^\smallfrown (\abar_i:i\in I_1)$ form an $M$-f.s.\ sequence over $D$,
	then $(\abar_i:i\in I_0)$ and $(\abar_i:i\in I_1)$ are mutually indiscernible over $b$ by Lemma~\ref{f.s.indisc}.  The uniqueness of the cut follows from Lemma \ref{bigind}(1).  If $i^*$ is either the maximal element of $I_0$ or the minimal element of $I_1$, it is immediate that $(\abar_i:i<i^*)$ and $(\abar_i:i>i^*)$ are mutually indiscernible over $\abar_{i^*}b$. The ``only if'' part about $i^*$ follows from Lemma \ref{bigind}(2).
	
	If there is no cut $I=I_0^\smallfrown I_1$ such that $(\abar_i:i\in I_0)^\smallfrown(b)^\smallfrown (\abar_i:i\in I_1)$ form an $M$-f.s.\ sequence over $D$, then there is $i^* \in I$ such that $(\abar_i:i<i^*)^\smallfrown (\abar_{i^*}b)^\smallfrown(\abar_i:i>i^*)$ is an $M$-f.s. sequence over $D$, and so by Lemma \ref{f.s.indisc}, $(\abar_i:i<i^*)$ and $(\abar_i:i>i^*)$ are mutually indiscernible over $a_{i^*}b$. The uniqueness of $i^*$ follows from Lemma \ref{bigind}(3).
\end{proof}

\section{Distality in hereditary classes and monadic distality}
This section contains some collapse results regarding (monadic) distality. It is not immediate that distality of a hereditary class implies monadic NIP, but it implies NIP and by \cite[Theorem 4.6]{EMonNIP}, this implies monadic NIP. The first collapse result is  that monadic distality is equivalent to monadic NIP and distality, mirroring the collapse of monadic NIP and stability to monadic stability. The second is the collapse of distality to monadic distality in hereditary classes, mirroring similar results for monadic stability and monadic NIP from \cite{EMonNIP}. After these collapse results, we begin studying distality and distal expansions in hereditary classes, showing that various natural definitions of distality coincide and proving that classes of bounded twin-width admit distal expansions in a way that also lifts to infinite models of their universal theory.

\subsection{Monadic distality}

\begin{definition} \label{def:dist}
A complete theory $T$ is \emph{distal} if for every indiscernible sequence $I = I_1 ^\smallfrown I_2 ^\smallfrown I_3$ such that $I_1$ has no maximum, $I_2$ no maximum or minimum, and $I_3$ no minimum, and for every $\abar, \bbar$, if $I_1 ^\smallfrown \abar  ^\smallfrown I_2 ^\smallfrown I_3$ and $I_1 ^\smallfrown I_2 ^\smallfrown \bbar ^\smallfrown I_3$ are each indiscernible, then so is $I_1 ^\smallfrown \abar  ^\smallfrown I_2 ^\smallfrown \bbar  ^\smallfrown I_3$.

An incomplete theory $T$ is {\em distal} if $Th(M)$ is distal for every $M\models T$.

A (possibly incomplete) theory $T$ is {\em monadically distal} if
$Th(M^+)$ is distal for every monadic expansion $M^+$ of every $M\models T$.
\end{definition}

Note that unlike many model-theoretic properties such as (monadic) stability/NIP, (monadic) distality is not preserved under reducts; for example, the theory of equality is not distal. Furthermore, the example below shows monadic distality is not preserved under passing to substructure. Nevertheless, we will recover various results for monadic distality paralleling those for monadic stability and monadic NIP.

 \begin{example}[(A theory $T$ that is monadically distal, but $T_\forall$ is not distal)] \label{ex:distsub}
\emph{Let $\LL = \set{U, V, <, E}$ be a relational language with $U, V$ unary and $<,E$ binary. Let $M$ be an infinite $\LL$-structure where $U, V$ partition the domain, each point is $E$-related to a unique point, any two $E$-related points are in different unary relations, and $<$ is a linear order on $U(M)$. It is easy to check that $M$ is monadically distal. However, the substructure $V(M)$ is a pure set and so is not distal.}
\end{example}

\begin{proposition}\label{prop:distobst}
	If $T$ is a theory that is monadically NIP but not monadically distal, then some $M \models T$ contains an infinite totally indiscernible set of singletons.
	\end{proposition}
\begin{proof}
	Passing to an appropriate completion of $T$ and renaming, we may assume $T$ is complete. Let $T'$ be a monadic expansion of $T$ that is not distal. Since $T$ is monadically NIP, so is $T'$, and so $T'$ is dp-minimal. By \cite[Corollary 9.19]{Guide}, if a theory is dp-minimal, then it is not distal if and only if some unrealized global invariant 1-type is generically stable. Let $p'(x)$ be such a type in $T'$, and consider a Morley sequence $\II'$ in $p'$ (as described after \cite[Lemma 2.23]{Guide}). Then by \cite[Theorem 2.29]{Guide}, $\II'$ is totally indiscernible, and thus so is the $T$-reduct $\II$.
	\end{proof}

\begin{corollary} \label{cor:distcoll1}
	Let $T$ be a theory. Then $T$ is monadically distal if and only if $T$ is distal and monadically NIP.
\end{corollary}

\begin{proof}
	The forward direction is clear, so suppose that $T$ is monadically NIP and distal. If $T$ is not monadically distal, then by the previous lemma some $M \models T$ contains an infinite totally indiscernible set, contradicting that $T$ is distal.
	\end{proof}

\subsection{Distality in hereditary classes}

\begin{definition}
    A \emph{hereditary class} is a class of structures closed under substructure.

    Given a class $\CC$ of structures, we define $Th(\CC)$ as $\bigcap_{M \in \CC} Th(M)$.
\end{definition}

In addition to Definition \ref{def:dist}, distality can be characterized in terms of strong honest definitions, which we now describe. This can be stated entirely at the level of a class of structures $\CC$, rather than having to look at models of $Th(\CC)$.

\begin{definition}
    Let $\CC$ be a class of structures. A formula $\phi(\xbar, \ybar)$ \emph{admits a strong honest definition in $\CC$} if there is a formula $\psi(\xbar, \zbar)$ such that for all $M \in \CC$, all finite $A \subset M$ with $|A| \geq 2$, and all $\bbar \in M^{|\xbar|}$, there is some $\dbar \in A^{|\zbar|}$ such that $M \models \psi(\bbar, \dbar)$ and all elements of $\psi(M, \dbar)$ realize the same $\phi$-type over $A$.

    We say $\phi$ \emph{admits a strong honest definition in $M$} if it does so in $\set{M}$. We say a class $\CC$ of structures \emph{admits strong honest definitions} if every formula admits strong honest definitions in $\CC$, and that a (possibly incomplete) theory $T$ \emph{admits strong honest definitions} if the class of models of $T$ does.
\end{definition}

We remark that the definable cell decompositions described in \cite[Definition 2.7]{chernikov2020cutting} are nearly the same as strong honest definitions, but allow for more refined quantitative results, which we will not be concerned with.

\begin{fact}[({\cite[Theorem 21]{chernikov2015externally}})] \label{fact:shd}
A complete theory $T$ is distal if and only if it admits strong honest definitions.
\end{fact}

The following lemma extends this fact to incomplete theories of the form $Th(\CC)$. The main point is $(1) \Ra (2)$ that from knowing only that $\phi$ admits a strong honest definition separately in each completion of $Th(\CC)$, we may deduce the uniformity result that there is a single strong honest definition for $\phi$ that works across all completions of $Th(\CC)$. This shows that there is a single robust definition of distality for a class of structures.

\begin{lemma} \label{lemma:dist defns}
	Let $\CC$ be a class of $\LL$-structures. The following are equivalent.
	\begin{enumerate}[(1)]
		\item $Th(\CC)$ is distal.
		\item $Th(\CC)$ admits strong honest definitions.
		\item $\CC$ admits strong honest definitions.
	\end{enumerate}
\end{lemma}
\begin{proof}
	The key point is that for a fixed formula $\phi(\xbar; \ybar)$ and a fixed formula $\psi(\xbar; \zbar)$, we may express by a set of sentences that $\psi$ is a strong honest definition for $\phi(\xbar; \ybar)$. Namely, for each $n\ge 1$, we may write a sentence $\theta_n$ stating that for every choice of $n$ elements $a_1, \dots, a_n$, for one of the finitely-many choices of $\psi$-definable families of sets over these elements, this family covers the whole structure and the truth value of $\phi(\xbar; \abar)$ is constant on each set in this family for each of the finitely many choices of $\abar \in \set{a_1, \dots, a_n}^{|\ybar|}$. Using this, the equivalences follow by standard compactness arguments.
	
	$(1) \Ra (2)$ Suppose there is some $\phi(\xbar; \ybar)$ such that there is no strong honest definition for $Th(\CC)$. So for each formula $\psi(\xbar, \zbar)$, there is some $M \models Th(\CC)$ such that $\psi$ fails over some parameter set $a_1, \dots, a_{n_\Psi}$. As above, we may write a sentence $\chi_\psi$ saying that there exist elements $a_1, \dots, a_{n_\Psi}$ over which $\psi$ fails. Note that by the ``standard coding trick'' (as in \cite[Lemma 2.5]{guingona2012uniform}) if there is some finite set $\Psi$ of formulas such that for each $M \in \CC$, some $\psi \in \Psi$ serves as a strong honest definition for $\phi$ in $M$, then $\phi$ admits a strong honest definition in $\CC$.  So $Th(\CC) \cup \set{\chi_\psi | \text{formulas } \psi(\xbar, \zbar)}$ is finitely satisfiable, and thus satisfiable. Thus there is some $M \models Th(\CC)$ such that there is no strong honest definition for $\phi(\xbar; \ybar)$ in $M$.
	
	$(2) \Ra (1)$ is immediate from Fact \ref{fact:shd}.
	
	$(3) \Ra (2)$ Suppose $\CC$ admits strong honest definitions. Fix a formula $\phi(\xbar; \ybar)$ and let $\psi(\xbar; \zbar)$ be strong honest definition. By assumption, every $\theta_n$ as described in the first paragraph is in $Th(\CC)$. 
	
	$(2) \Ra (3)$ We may express the assumption of $(2)$ via sentences $\theta_n$ as in the first paragraph. As $M \models Th(\CC)$ for each $M \in \CC$, the same strong honest definitions will work for $M$.
	\end{proof}

 The equivalences in the following two results are analogous to results for monadic NIP and monadic stability in \cite{EMonNIP}. For the following proposition, recall that Example \ref{ex:distsub} shows passing monadic distality from a theory to its universal part is not in general automatic.

\begin{proposition} \label{prop:distcoll}
	Let $\CC$ be a hereditary class of relational structures. Then the following are equivalent.
	\begin{enumerate}[(1)]
		\item $Th(\CC)_\forall$ is monadically distal.
		\item $Th(\CC)_\forall$ is distal.
		\item $Th(\CC)$ is monadically distal.
		\item $Th(\CC)$ is distal.
	\end{enumerate}
	\end{proposition}
\begin{proof}
	For the equivalences $(1) \iff (2)$ and $(3) \iff (4)$, note that if the theory $T$ in question is distal then it is NIP, and thus monadically NIP by \cite[Theorem 4.6]{EMonNIP}. Thus $T$ is monadically distal by Corollary \ref{cor:distcoll1}.
	
	Clearly $(1) \Ra (3)$, so we show $(3) \Ra (1)$. Suppose $Th(\CC)_\forall$ is not monadically distal. If it is not monadically NIP, then neither is $Th(\CC)$ by \cite[Proposition 2.5]{EMonNIP}, and we are finished, so assume it is monadically NIP. By Proposition \ref{prop:distobst}, there is some $M \models Th(\CC)_{\forall}$ containing an infinite totally indiscernible set $\II = \set{a_i | i \in \omega}$. By induction on quantifier complexity, any finite subset of $\II$ considered as structure is still totally indiscernible. Thus $\CC$ contains arbitrarily large totally indiscernible sets, and so by compactness $Th(\CC)$ has a model containing an infinite totally indiscernible set, contradicting distality.
	\end{proof}

\begin{lemma}
	Let $\CC$ be a hereditary class of relational structures. Then the following are equivalent.
	\begin{enumerate}[(1)]
		\item $Th(\CC)$ is monadically distal.
		\item For every monadic expansion $\CC^+$ of $\CC$, $Th(\CC^+)$ is distal.
	\end{enumerate}
\end{lemma}
\begin{proof}
	$(1) \Ra (2)$ Let $\CC^+$ be a monadic expansion of $\CC$, 
 i.e., a class where every element is a monadic expansion of some element of $\CC$. Let $M^+ \models Th(\CC^+)$ and let $M$ be the reduct of $M^+$ to the language of $\CC$. Then $M \models Th(\CC)$ is monadically distal, so $Th(M^+)$ is distal.
	
	$(2) \Ra (1)$ Let $M \models Th(\CC)$ and let $M^+$ be an arbitrary monadic expansion of $M$ to an $L^+$-structure.
 Let $\CC^+$ consist of  all $L^+$-structures
 $N^+$ whose $L$-reduct is in $\CC$.  As $\CC$ is hereditary, so is $\CC^+$.  Also, since $M\models Th(\CC)_\forall$, it follows that every finite $L^+$-substructure $N^+\subseteq M^+$ is an element of  $\CC^+$.  Thus,  $M^+\models Th(\CC^+)_\forall$.  But Proposition \ref{prop:distcoll}, combined with $(2)$, implies that $Th(\CC^+)_\forall$ is distal, so $Th(M^+)$ is distal as required.
	\end{proof}

Distality in hereditary classes very restrictive; for example, Ramsey's theorem implies that no hereditary graph class is distal. Since many of the good combinatorial properties of distal classes are preserved by passing to reducts, this motivates considering classes that have a distal expansion. As with distality, we are faced with choices about how much uniformity to require: we could ask only that every model of $Th(\CC)$ admit some distal expansion, or that there is a single expansion of $Th(\CC)$ with single uniform strong honest definition for every formula, or that there is an expansion $\CC^+$ of $\CC$ such that $Th(\CC^+$) is distal. There seems to be little reason to think these equivalent, but we will use the following implication.

\begin{lemma} \label{lemma:dist exp}
	Let $\CC$ be a class of relational $\LL$-structures. If $\CC$ admits an expansion to a class of $\LL^+$-structures $\CC^+$ so that $Th(\CC^+)_\forall$ is distal, then every model of $Th(\CC)_\forall$ admits a distal expansion to a model of $Th(\CC^+)_\forall$.
\end{lemma}
\begin{proof}
	Let $M \models Th(\CC)$. Consider the theory given by the atomic diagram of $M$ together with $Th(\CC^+)_\forall$. This is finitely satisfiable as witnessed by structures in $\CC^+$, and so has a model $N^+$, which embeds an $\LL^+$-expansion $M^+$ of $M$. Then $M^+ \models Th(\CC^+)_\forall$ is distal.
\end{proof}

\begin{lemma} \label{lem:ord dist}
Let $\CC$ be a class of structures such that $Th(\CC)$ is monadically NIP. If the structures in $\CC$ admit a $\emptyset$-definable linear order, then $Th(\CC)$ is distal.
\end{lemma}
\begin{proof}
	Since $\CC$ is monadically NIP, it is dp-minimal. Since the structures in $\CC$ are $\emptyset$-definably linearly ordered, they are distal by \cite[Example 9.20]{Guide}.
\end{proof}

We thus have a dichotomy for monadically NIP hereditary classes: either the structures in $\CC$ admit an order-expansion (i.e., a class where every element is an expansion of some element of $\CC$ by a linear order) that remains monadically NIP, which is also a distal expansion, or no distal expansion of $\CC$ can have a $\emptyset$-definable linear order. In the case where $\CC$ consists of binary relational structures, the classes that admit an order-expansion remaining monadically NIP coincides with the classes of \emph{bounded twin-width}. The twin-width of a binary structure has a combinatorial definition (see \cite{bonnet2021twin}) that we will not use except in Example \ref{ex:forest}, and a class $\CC$ has bounded twin-width if there is a uniform finite bound on the twin-width of all members. 

\begin{lemma}
    Let $\CC$ be a class of finite structures in a binary relational language. Then $\CC$ has bounded-twin width if and only if $\CC$ admits an order-expansion $\CC^<$ that is monadically NIP.
\end{lemma}
\begin{proof}
    Since bounded twin-width and monadic NIP are both preserved by closing under substructure, we may assume that $\CC$ and $\CC^<$ are hereditary. By \cite[Theorem 3]{bonnet2024twin}, a hereditary class of ordered binary structures is monadically NIP if and only if it has bounded twin-width.

    $(\Ra)$ Suppose $\CC$ has bounded twin-width. Then it admits an order-expansion $\CC^<$ with bounded twin-width (see \cite[Fact 2]{simon2021ordered}). Then $\CC^<$ is monadically NIP, and thus so is $\CC$.

    $(\La)$ Suppose $\CC$ admits a monadically NIP order-expansion $\CC^<$. Then $\CC^<$ has bounded twin-width, which is preserved by reducts, so $\CC$ has bounded twin-width.
\end{proof}

 Classes with bounded twin-width include, for example, the class of planar graphs, and more generally every graph class determined by forbidden minors (other than the class of all graphs) \cite{bonnet2021twin}. The second part of the following result generalizes \cite{przybyszewski2023distal}, which showed a strong honest definition for the edge relation in graph classes of bounded twin-width, although \cite{przybyszewski2023distal} also gives explicit bounds (as a function of the twin-width) on the complexity of the cell decompositions arising from the honest definition, which we do not.

Putting together the various results of this section, we obtain the following.

\begin{theorem} \label{thm:distexp}
	Let $\CC$ be a hereditary class of binary relational structures with bounded twin-width. Then $\CC$ admits a distal order-expansion, and thus so does every model of $Th(\CC)_\forall$. In the distal order-expansion of $\CC$, every formula admits a strong honest definition.
\end{theorem}
\begin{proof}
	Let $\CC$ be as assumed, and let $\CC^<$ be an order-expansion that is monadically NIP. By \cite[Proposition 2.5]{EMonNIP}, every model of $Th(\CC^<)_\forall$ is monadically NIP and so is distal by Lemma \ref{lem:ord dist}. By Lemma \ref{lemma:dist exp}, every model of $Th(\CC)_\forall$ admits a distal order-expansion. Since $Th(\CC^<)$ is distal, every formula admits a strong honest definition by Lemma \ref{lemma:dist defns}.
	\end{proof}

\begin{corollary} \label{cor:planar}
	Let $G$ be an infinite planar graph. Then $G$ admits a distal order-expansion.
\end{corollary}

The question of whether every planar graph has a distal expansion was raised to us by Artem Chernikov, and is also asked in \cite[Question 1.10]{de2025planar}.

The distal expansion of such an infinite $G$ is produced by the compactness argument of Lemma \ref{lemma:dist exp}, so in general even if we explicitly know the distal order-expansion of $Age(G)$, we do not explicitly obtain the distal order-expansion of $G$. However, in some cases we can obtain an explicit expansion, such as in the following example.

\begin{example}[(A distal expansion of the everywhere infinite forest)] \label{ex:forest}
\emph{We produce a distal expansion of the everywhere infinite forest, i.e. of the theory of an infinite acyclic graph where every point has infinite degree. For a finite forest, it is easy to describe a contraction sequence witnessing bounded twin-width (see \cite[Lemma 5]{berge2021deciding}), which can be used as described in the proof of \cite[Fact 2]{simon2021ordered} to construct the following distal order-expansion for the class of finite forests. Each finite tree is ordered by choosing an arbitrary root, and then linearly ordering the vertices so that they satisfy the following condition $(*)$: for every vertex $x$, the set consisting of $x$ and its descendants forms a convex set in the order, with $x$ as either the first or last point. A finite forest is then ordered by separately ordering each tree to satisfy $(*)$, and then making each tree a convex set. The resulting class of finite ordered forests has bounded twin-width and so is distal, and the same is true for its closure under substructure.}

\emph{We can choose an arbitrary root for each tree in the everywhere infinite forest and then construct from the root upwards a linear order that satisfies $(*)$, and complete the order to make each tree convex. This will be a distal expansion, since every finite substructure of the ordered forest will be a substructure of a finite forest ordered as above.}

\emph{Note that the proof that this is a distal expansion is reduced to combinatorics on finite forests, and does not rely on fine analysis of definability or indiscernible sequences in the infinite model.}
\end{example}

Since distality is supposed to capture order-like behavior, it may be that the second possibility in the dichotomy discussed above never occurs for a class $\CC$ admitting a distal expansion (i.e. it cannot be that no distal expansion of $\CC$ has a definable linear order). In the case of binary structures, this would yield a converse to the previous theorem, showing every hereditary class admitting a distal expansion must have bounded twin-width.

\begin{question}
	Let $\CC$ be a hereditary class of relational structures. If $\CC$ admits an expansion $\CC^+$ such that $Th(\CC^+)$ is distal, can we take the additional relations in $\CC^+$ to be solely a linear order (equivalently, to include a linear order)?
\end{question}

There is also the analogous question in the setting of complete theories.

\begin{question}
	Let $T$ be a complete monadically NIP theory and $T^+$ a distal expansion of $T$. Does $T^+$ admit an order-expansion that remains distal (or at least NIP)?
\end{question}

In the first question, the fact that $\CC$ admits a distal expansion implies that $\CC$ is monadically NIP, but this must be made explicit in the second question, since the answer is negative without this assumption. The theory $T$ described in \cite[Sections 2.1-2.3]{aschenbrenner2022distality} is a distal expansion of the theory of infinite-dimensional $\mathbb{F}_p$-vector spaces (where the expansion essentially adds a valuation), but \cite[Theorem 2.1]{shelah2012adding} shows that every order-expansion of an infinite-dimensional $\mathbb{F}_p$-vector space has IP.

\section{Not interpreting an infinite group}

In this section, we show a strengthening of the statement that if $T$ is monadically NIP, then $T$ does not interpret an infinite group. We then use this to show the existence of a (monadically) stable theory that has no distal expansion but does not interpret an infinite group. We begin with the following weakening of the notion of interpretation.

\begin{definition}[(\cite{trace})]
	Let $M, N$ be structures (in possibly different languages). Let $\tau \colon N \to M^m$ be an injection. We say that {\em $M$ trace defines $N$ (via $\tau$)} if for every $N$-definable subset $X$ of $N^n$ there is an $M$-definable subset $Y$ of $M^{mn}$ such that:
	$$\text{for all } a_1, \dots, a_n \in N, (a_1, \dots, a_n) \in X \iff (\tau(a_1), \dots, \tau(a_n)) \in Y $$
	
	Let $T, T^*$ be theories. Then $T$ {\em trace defines} $T^*$ if every $T^*$-model is trace definable in a $T$ model, and $T$ {\em trace defines} $M$ if some $T$-model trace defines $M$.
\end{definition}

In particular, if $M$ interprets $N$ then $M$ trace defines $N$.
However, if  $M$ trace defines $N$ via $\tau:N\rightarrow M^m$, 
then the image $\tau[N]$ need not be a definable subset of $M^m$.

In \cite{trace}, it was conjectured that monadically NIP theories cannot trace define an infinite group. We confirm this with the more general Proposition \ref{prop:nointerp}.

\begin{definition}
	Let $(G, \cdot)$ be a set equipped with a binary function. Then $(G, \cdot)$ is a {\em cancellative magma} if it satisfies the following conditions: for every $a,b \in G$, there is at most one $x \in G$ such that $a \cdot x = b$ and at most one $y \in G$ such that $y \cdot a = b$.	
\end{definition}

The notion of dp-rank is characterized in \cite[Theorem 4.18]{Guide}, and dp-minimality corresponds to the case of dp-rank 1. The only result about finite dp-rank that we need will be recalled within the following proof.

The following argument is similar to that given in \cite[A.6.9]{hodges1993model} that a colored linear order cannot interpret an infinite group (first given in \cite{poizat1987propos}, in the more general setting of ``local theories'').

\begin{proposition}  \label{prop:nointerp}
	Let $T$ be a theory with finite dp-rank and endless indiscernible triviality. Then $T$ does not trace define an infinite cancellative magma.
\end{proposition}
\begin{proof}
	Let $M \models T$ and suppose $M$ trace defines an infinite cancellative magma $(G, \cdot)$ via $\tau \colon G \to M^n$.  Let $M^+$ be the expansion of $M$ by an $n$-ary relation $R$ naming $\tau(G)$. By passing to an elementary extension (and renaming $M$ and $G$), we may suppose there is an indiscernible sequence $\II = (\abar_i \in R(M^+): i \in \Q)$. We now return to working in the original language.
	
	Since $T$ has finite dp-rank, there is some fixed $K \in \omega$, depending only on $n$ and the dp-rank of the theory, such that for every $\pbar\in M^n$,  there is a finite subset $F\subseteq \R$ with $|F|<K$ such that in the induced partition of  $\Q\setminus F=\bigsqcup \set{I_i |i\le |F|}$ into open convex sets, 
 each of the subsequences $\II_i:=(\abar_q:q\in I_i)$ is indiscernible over
 $\pbar$.
 
 Let $\bbar \in \tau(G)$ correspond to the product $(((\abar_1 \cdot \abar_2) \cdot \abar_3) \cdot {\dots}  \cdot \abar_{K})$ (using the ``pushforward'' of $\cdot$ to $\tau(G)$). From above, find
 a finite subset $F\subseteq\R$ with $|F|<K$ for this $\bbar$ and choose an integer $i^*$,
 $1\le i^*\le K$ with $i^*\not\in F$.  
 Choose an open interval
 $J\subseteq \Q$ for which $J\cap \set{1,\dots,K}=\set{i^*}$ and for which $\J:=(\abar_j:j\in J)$ is indiscernible over $\bbar$.
 Since $\II$ was chosen to be indiscernible over $\emptyset$, 
 $\J=(\abar_j:j\in J)$ is indiscernible over $\set{\abar_q | q\not\in J}$.  
 By endless indiscernible triviality, $\J$ would be indiscernible over
 $\bbar\cup\set{\abar_i | 1\le i\le K, i\neq i^*}$.
 But this is impossible.  
 Choose any $j\in J\setminus\set{i^*}$.  It follows from the indiscernibility that
 $\bbar'$, which is the same product but with $\abar_{i^*}$ replaced by $\abar_j$, must be equal to $\bbar$.  By repeated applications of cancellation, we would obtain $\abar_j=\abar_{i^*}$, which is a contradiction.
 \end{proof}

	\begin{corollary} \label{cor:nogroup}
		If $T$ is a monadically NIP theory, then $T$ does not trace define an infinite cancellative magma.
	\end{corollary}

 	We remark that monadically NIP theories can have non-trivial binary functions, so long as they are not cancellative. For example, the theory of dense meet-trees in \cite[Section 2.3.1]{Guide} is monadically NIP by \cite{parigot1982theories}.

\begin{question}
    Can the hypothesis of finite dp-rank  in Proposition \ref{prop:nointerp} be weakened to NIP, or even removed altogether?
\end{question}

The following characterizes which o-minimal theories are monadically NIP.  It is noteworthy that for this class of theories, indiscernible triviality and endless indiscernible triviality coincide.

\begin{definition}  An o-minimal structure $(M,<,\dots)$ whose ordering is dense is {\em non-trivial} if there is a non-empty open interval $I\subseteq M$ and a definable, continuous function $f \colon I\times I\rightarrow M$ that is strictly monotone in both variables.
\end{definition}
 
	\begin{proposition}
		The following are equivalent for  a complete, $o$-minimal theory $T$ whose underlying order is dense. 
  \begin{enumerate}[(1)]
      \item $T$ is trivial.
      \item  $T$ has indiscernible  triviality.
      \item  $T$ has endless indiscernible triviality.
      \item $T$ is monadically NIP.
  \end{enumerate}
	\end{proposition}
	\begin{proof}
  $(1)\Rightarrow(2)$  Suppose that $T$ is o-minimal and trivial.  It follows from Lemmas~2.1 and 2.2 of 
  \cite{mekler1992dedekind} 
  that $T$ is {\em binary}, i.e. for every $M\models T$, the type of any $n$-tuple of elements is determined by the union of the types of 
  its 2-element subsequences.
  Clearly, any binary theory $T$ must satisfy indiscernible triviality. 

  $(2)\Rightarrow(3)$ is immediate. As $T$ is o-minimal, it is dp-minimal,  and so $(3)\Rightarrow(4)$ follows from
  Theorem~\ref{thm:mNIPchar}.

  $(4) \Ra (1)$ Assume that $(1)$ fails, i.e., there is a model $M\models T$ and a definable, continuous function $f\colon I^2\rightarrow M$ that is strictly monotone in both variables.
  There are two arguments to show that $T$ is not monadically NIP.  On the one hand, note that one can construct, element by element, an infinite $G\subseteq I$ for which $f\mr{G^2}$ is 1-1.  [If $G_n\subseteq I$ is finite and $f\mr{G_n^2}$ is 1-1, then $f\mr {(G_n\cup \{g\})^2}$ is 1-1 for all but finitely many $g\in I$.]  This directly describes a tuple-coding configuration as in Definition \ref{def:coding}, contradicting monadic NIP.  Alternately, quoting Theorem~1.1 of \cite{KobiSergei}, we see that $M$ trace defines an infinite group, hence $T$ is not monadically NIP by Corollary~\ref{cor:nogroup}.
	\end{proof}

    Together with \cite{KobiSergei} and \cite[Theorem 0.7]{trace}, this result shows that endless indiscernible triviality at least agrees with the correct notion of triviality for the Zilber dichotomy among o-minimal theories using trace definability. However, we think it unlikely to give a dichotomy for dp-minimal theories. The leveled tree structure $M_{C, \rm{lev}}$ appearing in \cite[Lemma 2.8, Proposition 2.9]{almazaydeh2024omega} is dp-minimal but does not have endless indiscernible triviality. Since $M_{C, \rm{lev}}$ is homogeneous in a finite relational language, \cite{macpherson1991interpreting} shows it does not interpret an infinite group, and we think it unlikely it trace defines an infinite group.

	We now turn to the existence of a monadically stable theory with no distal expansion that does not interpret an infinite group. This example seems to be folklore to some extent (\cite{gentle} credits Pierre Simon with asking for the relevant graph class), and for its analysis it is already enough to know that monadically stable theories have trivial forking  \cite{baldwin1985second}. But since it has not appeared explicitly in the literature, we record it here as a corollary of Corollary \ref{cor:nogroup}.
	
	\begin{definition}
		Let $\CC$ be a class of graphs. Then $\CC$ is {\em somewhere dense} if for some $d$, there are $d$-subdivisions of arbitrarily large complete graphs appearing as (not necessarily induced) subgraphs of graphs in $\CC$. Otherwise, $\CC$ is {\em nowhere dense}. We will say a single infinite graph $G$ is nowhere dense if $\set{G}$ is.
	\end{definition}
	
	In \cite{podewski1978stable}, it is shown that if $G$ is nowhere dense (there called superflat), then it is monadically stable. We now introduce a combinatorial condition that is necessary for $M$ to have a distal expansion \cite[Corollary 4.10]{chernikov2018regularity}.
	
	\begin{definition}
		Given a class of structures $\CC$ and a formula $\phi(\xbar, \ybar)$, we say {\em $\phi$ has the strong Erd\H os-Hajnal property in $\CC$} there is a fixed $\delta > 0$ such that for every $M \in \CC$ and finite $A \subset M^{|\xbar|}, B \subset M^{|\ybar|}$, there are $A' \subset A, B' \subset B$ such that $|A'| \geq \delta |A|, |B'| \geq \delta |B|$ and either $A' \times B' \subset \phi(M)$ or $(A' \times B') \cap \phi(M) = \emptyset$.
	\end{definition}

	\begin{lemma}[({\cite[Theorem 50]{gentle}, \cite[Proposition 5]{EH}})] \label{lemma:nonEH}
		There is a nowhere dense class of finite graphs $\mathscr{R}$ such that the edge relation does not have the strong Erd\H os-Hajnal property.
	\end{lemma}

 We remark that the class constructed in \cite[Theorem 50]{gentle} comes from taking expander graphs of increasing expansion, while the class in \cite[Proposition 5]{EH} comes from a probabilistic construction.
	
	\begin{corollary} \label{cor:nodist}
		There is a monadically stable theory without a distal expansion that does not trace define an infinite cancellative magma (so in particular, does not interpret an infinite group).
	\end{corollary}
	\begin{proof}
		Let $\mathscr{R}$ be a graph class as in Lemma \ref{lemma:nonEH}, let $\sqcup \mathscr{R}$ be the disjoint union of the graphs in $\mathscr{R}$, and let $T := Th(\sqcup \mathscr{R})$. Then $\sqcup \mathscr{R}$ is nowhere dense, and so $T$ is monadically stable, and so does not trace define an infinite cancellative magma by Corollary \ref{cor:nogroup}. The edge relation in $\sqcup \mathscr{R}$ also does not have the strong Erd\H os-Hajnal property, and so $T$ does not admit a distal expansion.
	\end{proof}

 \bibliographystyle{alpha}
\bibliography{Bib}

\affiliationone{
   Samuel Braunfeld\\
   Informatick\'y \'ustav Univerzity Karlovy, Malostransk\'e n\'am. 25, 118 00 Prague\\
   Czech Republic
   \email{sbraunfeld@iuuk.mff.cuni.cz} \\
   The Czech Academy of Sciences, Institute of Computer Science, Pod Vod\'arenskou v\v{e}\v{z}\'i 2, 18200 Prague\\
   Czech Republic
   \email{braunfeld@cs.cas.cz}}
   \affiliationtwo{
   Michael C. Laskowski\\
   Department of Mathematics\\ University of Maryland \\College Park, MD 20742 \\
   USA
   \email{laskow@umd.edu}}

\end{document}